\documentclass[11pt,a4paper,draft,reqno]{amsart}
\usepackage[headings]{fullpage}
\usepackage{amssymb,amscd,hyperref}
\usepackage{amsthm}
\usepackage[alphabetic, nobysame]{amsrefs}
\usepackage[all]{xy}
\usepackage{color}
\numberwithin{equation}{section}
\newtheorem{thm}{Theorem}[section]
\newtheorem{lem}[thm]{Lemma}
\newtheorem{prop}[thm]{Proposition}
\newtheorem{cor}[thm]{Corollary}

\theoremstyle{definition}

\newtheorem{conj}[thm]{Conjecture}
\theoremstyle{remark}
\newtheorem{rem}[thm]{Remark}
\newtheorem*{acknowledgements}{Acknowledgements}
\DeclareMathOperator{\colim}{colim}
\def\Z{\mathbb{Z}}
\def\Q{\mathbb{Q}}
\def\F{\mathbb{F}}
\def\etwo{E}
\newcommand{\itemr}[1]{\item[\rm{(#1)}]}

\def\thh{\mathrm{THH}}
\def\bthh{\overline{\thh}}

\def\ra{\rightarrow}
\def\leq{\leqslant}
\def\geq{\geqslant}

\def\HH{\mathsf{HH}}
\newcommand{\xr}{\xrightarrow}

\def\ie{\emph{i.e.}}
\def\eg{\emph{e.g.}}
\def\id{\mathrm{id}}

\begin{document}
\title[Towards THH of Johnson-Wilson spectra]{%
Towards topological Hochschild homology of Johnson-Wilson spectra}
\author{Christian Ausoni}
\address{LAGA (UMR7539), Institut Galil\'ee, 
Universit\'e Paris 13 Sorbonne-Paris-Cit\'e, 99 avenue J.-B. Cl\'ement,
93430 Villetaneuse, France}
\email{ausoni@math.univ-paris13.fr}
\urladdr{http://www.math.univ-paris13.fr/~ausoni/}

\author{Birgit Richter}
\address{Fachbereich Mathematik der Universit\"at Hamburg,
Bundesstra{\ss}e 55, 20146 Hamburg, Germany}
\email{birgit.richter@uni-hamburg.de}
\urladdr{http://www.math.uni-hamburg.de/home/richter/}
\date{\today}
\keywords{Topological Hochschild homology, Johnson-Wilson spectra,
  $E_\infty$-structures on ring spectra, chromatic squares}
\subjclass[2000]{55P43, 55N35}

\begin{abstract}
We offer a complete description of $\thh(E(2))$ under the assumption
that the Johnson-Wilson spectrum $E(2)$ at a chosen odd prime carries
an $E_\infty$-structure. We also place $\thh(E(2))$ in a cofiber
sequence $E(2) \ra \thh(E(2))\ra  \bthh(E(2))$ and describe $\bthh(E(2))$
under the assumption that $E(2)$ is an $E_3$-ring spectrum. We state general
results about the $K(i)$-local behaviour of $\thh(E(n))$ for all $n$
and $0 \leq i \leq n$. In particular, we compute $K(i)_*\thh(E(n))$.
\end{abstract}
\maketitle

\section{Introduction}

The first Johnson-Wilson spectrum $E(1)$ at a prime $p$ is the Adams summand of
$p$-local periodic complex topological $K$-theory $KU_{(p)}$. It is known that
it carries a
unique $E_\infty$-structure \cite{mccs,br}, thus $\thh(E(1))$ is a
commutative $E(1)$-algebra spectrum. McClure and Staffeldt show that
the unit  map
$E(1) \ra \thh(E(1))$ is a $K(1)$-local equivalence, hence its cofiber
$\bthh(E(1))$ is a rational spectrum. It is easy to calculate the
rational homology of $\thh(E(1))$ as
$$ H\Q_*\thh(E(1)) \cong \Q[v_1^{\pm 1}] \otimes_\Q
\Lambda_\Q(dv_1)$$
using the B\"okstedt spectral sequence with $E^2$-term 
$$ E^2_{*,*} = \HH_{*,*}^\Q(\Q[v_1^{\pm 1}]).$$ 
There is a map
$$ \Sigma^{2p-1} E(1) \ra  \thh(E(1)) \ra \bthh(E(1))$$
that factors through
$\Sigma^{2p-1} E(1)_\Q\ra \bthh(E(1))$ since $\bthh(E(1))$ is rational, and
that is defined such that the latter map is an equivalence detecting the
$H\Q_*E(1)$-summand generated by
$dv_1$. Since the unit map $E(1)\ra\thh(E(1))$ splits, this yields a splitting
\cite[Theorem 8.1]{mccs} 
$$ \thh(E(1)) \simeq E(1) \vee \Sigma^{2p-1} E(1)_\Q$$
as $E(1)$-modules. This computation was also carried out for
$KU_{(p)}$~\cite{ausoni-thh}, and pushed further to provide formulas for
$\thh(KU)$ as a commutative $KU$-algebra by Stonek~\cite{stonek}.

In this paper, we consider
the higher Johnson-Wilson spectrum $E(n)$ with coefficient ring
$$
E(n)_*=\Z_{(p)}[v_1,\dots,v_{n-1},v_n,v_n^{-1}]
$$
for an arbitrary value of $n\geq1$ and $p$ an odd prime. A main motivation here is to investigate whether
the spectrum $\thh(E(n))$ also splits into copies of $E(n)$ and its lower
chromatic localizations, generalizing McClure and Staffeldt's intriguing
transchromatic result.

As a first step, we compute 
the Hochschild homology $\HH_*^{K(i)_*}(K(i)_*E(n))$ of $K(i)_*E(n)$, 
where $K(i)$ is the $i$th Morava $K$-theory, for
$0\leq i\leq n$, at an odd prime, see Theorem~\ref{thm:hh}. 
We shy away from the prime $2$ because Morava $K$-theory is not homotopy
commutative at the prime $2$.
Theorem~\ref{thm:hh} yields a computation of $K(i)_*\thh(E(n))$ under the
modest assumption that $E(n)$ admits an $E_3$-structure.

We then focus on $E(2)$, and show 
in Theorem~\ref{thm:thhofe2} 
that under the same commutativity assumption 
$\thh(E(2))$ sits in a cofiber sequence
$$
E(2)\to\thh(E(2))\to
\Sigma^{2p-1}L_1E(2)\vee
\Sigma^{2p^2-1}E(2)_\Q\vee
\Sigma^{2p^2+2p-2}E(2)_\Q\,,
$$
where $L_1E(2)$ denotes the Bousfield localization of $E(2)$ with respect to
$E(1)$. If the unit $E(2)\to \thh(E(2))$ splits, we then get a decomposition of
$\thh(E(2))$ into four
summands, a higher analogue of McClure-Staffeldt's formula for $\thh(E(1))$. 

\begin{rem}
To study $\thh(E(n))$ by means of the B{\"o}kstedt spectral sequence, we need
sufficient commutativity of $E(n)$. In this remark, we summarize what is known
about multiplicative structures on $E(n)$ and related spectra.
Basterra and Mandell showed~\cite{bm13} that the Brown-Peterson spectrum $BP$
admits an $E_4$ structure. 
The Johnson-Wilson spectra $E(n)$ are built out 
of the $BP\langle n \rangle=BP/(v_i|i\geq n+1)$ by inverting $v_n$. 
In \cite[Theorem 1.1.2]{lawson} Tyler Lawson shows that the Brown-Peterson spectrum $BP$
and the spectra $BP\langle n \rangle$ for $n \geq 4$ at 
the prime $2$ do not possess an
$E_{12}$-structure. Andrew
Senger~\cite[Theorem 1.2]{senger}  extends Lawson's result to odd primes 
$p$, and shows 
that $BP$ and the $BP\langle n\rangle$'s (for $n \geq 4$) do not have
an $E_{2(p^2+2)}$-structure. In particular, the $BP\langle n\rangle$'s are not
$E_\infty$-ring spectra at any prime for $n \geq 4$.  
Hence if $E(n)$ actually possesses an $E_\infty$-structure for $n \geq 4$,
then this structure does not come from one on the $BP\langle n\rangle$'s. 
In \cite[Proposition 8.2]{r-bp} it is proven that $E(n)$ at a prime $p$
possesses at least a $(2p-1)$-stage structure. It is unclear how such a
structure relates to the $E_n$-hierarchy, but Barwick conjectures
\cite[p.~1948]{barwick} that a $(2p-1)$-stage structure corresponds to an
$A^{2p-1}_{2p}$-structure which in turn is a filtration piece of an
$E_{2p-1}$-structure. 

At the prime $2$, Lawson and Naumann~\cite{ln1} show that there is an
$E_\infty$-model of $BP\langle 2\rangle$ and Hill and Lawson~\cite{hl} prove
that $BP\langle 2\rangle$ at the prime $3$ possesses a model as an 
$E_\infty$-ring spectrum. 
With~\cite[Theorem A.1]{mnn} this yields
$E_\infty$-structures on the corresponding Johnson-Wilson spectra
$E(2)$ at these primes. 
\end{rem}

\begin{acknowledgements}
The first named author acknowledges support from the project
ANR-16-CE40-0003 ChroK. 
The second named author thanks the University of Paris 13 for its
hospitality and for the 
possibility of a research stay as \emph{professeur invit\'ee}. Both 
authors benefited from a stay at the Hausdorff Institute for Mathematics in
Bonn during the Trimester Program on \emph{$K$-theory and Related Fields}.   

We thank Paul Goerss for a crucial hint that simplified our
original \'etaleness argument, and
Agn{\`e}s Beaudry, Gerd Laures, Mike Mandell, John Rognes, and Vesna Stojanoska 
for helpful comments. 
\end{acknowledgements}

\section{Rationalized  $E(n)$}
For  $n \geq 1$ the homotopy algebra of $L_{K(0)}E(n) = E(n)_{\Q}$ is
$\Q[v_1, \ldots, 
v_{n-1}, v_n^{\pm 1}]$ and its algebra of cooperations is 
$$ \pi_*(E(n)_{\Q} \wedge E(n)_{\Q}) \cong \pi_*E(n)_{\Q} \otimes_\Q
\pi_*E(n)_{\Q} \cong \Q[v_1, \ldots,
v_{n-1}, v_n^{\pm 1}, v'_1, \ldots, v'_{n-1}, {v'_n}^{\pm 1}]. $$ 
This implies the following result. 
\begin{lem}
There is a unique $E_\infty$-ring
structure on $E(n)_{\Q}$ for all $n \geq 1$. 
\end{lem}
\begin{proof}
The obstruction groups for such an  $E_\infty$-ring
structure on $E(n)_{\Q}$ are contained in the Gamma cohomology groups of
$\pi_*(E(n)_{\Q} \wedge E(n)_{\Q})$ as a $\pi_*E(n)_{\Q}$-algebra
\cite[Theorem 5.6]{robinson}. As
we work in characteristic zero, Gamma cohomology agrees with
Andr\'e-Quillen cohomology \cite[Corollary 6.6]{RobWh}. The algebra
$\Q[v_1, \ldots, 
v_{n-1}, v_n^{\pm 1}, v'_1, \ldots, v'_{n-1}, {v'_n}^{\pm 1}]$ is smooth
over $\Q[v_1, \ldots,
v_{n-1}, v_n^{\pm 1}]$ and therefore Andr\'e-Quillen cohomology is
concentrated in cohomological degree zero where it consists of 
derivations. The obstructions for existence and uniqueness of an $E_\infty$-ring
structure on $E(n)_{\Q}$ are concentrated in degrees bigger than
zero. 
\end{proof}
As $E_\infty$-ring structures can be rigidified to commutative
ring structures (see \eg,  \cite[II.3]{ekmm}), we pass to the world of
commutative ring spectra from now on. 

Topological Hochschild homology of a ring spectrum $A$ can be modelled
as the geometric 
realization of a simplicial spectrum. Using the inclusion of the $1$-skeleton,
McClure and Staffeldt \cite[\S
3]{mccs} construct a map 
\begin{equation}\label{sigma}
\sigma\colon  \Sigma A \ra \thh(A)\,.
\end{equation}
For a commutative ring spectrum $A$ the multiplication maps from $A^{\wedge
  n+1}$ to $A$ give rise to a map of commutative $A$-algebra spectra from 
$\thh(A)$ to $A$. Composing this map with the map $A \ra \thh(A)$
gives the identity, hence we obtain a splitting of $A$-modules
\begin{equation*}
\thh(A) \simeq A \vee \bthh(A)
\end{equation*}
where $\bthh(A)$ is the cofiber. The latter spectrum
inherits the structure of a non-unital commutative $A$-algebra. In our
case this implies the following result. 
\begin{cor} \label{cor:cofnuca}
  The topological Hochschild homology of $E(n)_\Q$ splits, as an
  $E(n)_\Q$-module, as 
\[ \thh(E(n)_\Q) \simeq E(n)_\Q \vee \bthh(E(n))_\Q\] 
where $\bthh(E(n))_\Q$ is the cofiber of the unit map $E(n)_\Q \ra
\thh(E(n)_\Q) \simeq \thh(E(n))_\Q$. Moreover, the spectrum $\bthh(E(n))_\Q$ is a non-unital commutative
$E(n)_\Q$-algebra. 
\end{cor}
In the sequel, we follow Loday~\cite[Definition~E.1]{loday} for the definition
of {\'e}tale algebras.
It is straightforward to calculate the topological Hochschild homology
of $E(n)_\Q$. 
\begin{prop}
\begin{equation} \label{eq:thhenq}
\pi_*\thh(E(n))_\Q \cong \Q[v_1, \ldots, 
v_{n-1}, v_n^{\pm 1}]  \otimes \Lambda_\Q(dv_1, \ldots, dv_n)
\end{equation}
with $|dv_i| = 2p^i-1$. 
\end{prop}
\begin{proof}
The B\"okstedt spectral sequence for $\pi_*(\thh(E(n))_\Q) \cong
H\Q_*\thh(E(n))$ is of the form
$$ E^2_{*,*} = \HH_{*,*}^\Q(\pi_*E(n)_\Q) \Rightarrow
\pi_*(\thh(E(n))_\Q).$$
As $\Q[v_1, \ldots, 
v_{n-1}, v_n^{\pm 1}]$ is \'etale over  $\Q[v_1, \ldots, 
v_{n-1}, v_n]$ and as $\Q[v_1, \ldots, 
v_{n-1}, v_n]$ is smooth, we get 
$$ \HH_{*,*}^\Q(\pi_*E(n)_\Q) \cong \Q[v_1, \ldots, 
v_{n-1}, v_n^{\pm 1}] \otimes \Lambda_\Q(dv_1, \ldots, dv_n)$$
with $dv_i$ having homological degree one and internal degree
$2p^i-2$. As the B\"okstedt spectral sequence is multiplicative and as
the algebra generator cannot support any differentials for degree reasons, the spectral
sequence collapses at $E^2$. There are no multiplicative extensions
and hence we get the result. 
\end{proof}
\begin{rem}
As we work rationally,  $\thh(E(n))_\Q$ is a commutative $H\Q$-algebra
spectrum and hence corresponds to a commutative differential graded
$\Q$-algebra (see \cite{shipley} or \cite{rs}). 
\end{rem}

\section{$K(i)_*E(n)$ and $K(i)_*\thh(E(n))$}
In the following we assume that $p$ is an odd prime, and that $n$ and $i$ are integers with
$1\leq i\leq n$. 
The Hopf algebroid $(BP_*, BP_*BP)$ represents the groupoid of strict 
isomorphisms of $p$-typical formal group laws \cite{landweber} (see
also \cite[Theorem A2.1.27]{ravenelcob}). 
There are isomorphisms of graded $\Z_{(p)}$-algebras
$$
BP_*\cong \Z_{(p)}[v_1,v_2,\dots]\ \ \textup{and}\ \ BP_*BP\cong BP_*[t_1,t_2,\dots]\,,
$$
where $|v_i|=|t_i|=2(p^i-1)$. By convention $v_0=p$ and $t_0=1$.
The $i$th Morava $K$-theory $K(i)$ is complex oriented, and
its formal group law $F_i$ (the Honda formal group law) corresponds to the map
$BP_*\to K(i)_*=\F_p[v_i^\pm]$ sending $v_i$ to $v_i$ and $v_k$ for $k\neq i$ to zero.
The $p$-typical formal group law $G_n$ over $E(n)_*$ comes
from the map $BP_* \ra E(n)_*$ that kills all $v_i$ with $i > n$ and inverts 
$v_n$. 
Since $E(n)$ is a Landweber exact homology theory, we obtain an isomorphism
\begin{equation} \label{eq:kien}
K(i)_*E(n) \cong K(i)_* \otimes_{BP_*} BP_*BP \otimes_{BP_*} E(n)_*.
\end{equation}
Note that $K(i)_*E(n)$ is trivial for $i > n$ and that the Bousfield class of
$E(n)$, $\langle E(n)\rangle$, is $\langle K(0) \vee \ldots \vee K(n)\rangle$. 

We first treat the case $i=n$.

\begin{prop} \label{prop:knlocal}
For all $n \geq 1$ the canonical map $E(n) \ra \thh(E(n))$ is a
$K(n)$-local equivalence. 
\end{prop}
\begin{proof}
The algebra $K(n)_*E(n)$ is known as $\Sigma(n)$ and it is of the
form  
$$ K(n)_*[t_1,t_2,\ldots]/(v_nt_i^{p^n}-v_n^{p^i}t_i, i \geq 1), $$ 
see \cite[6.1.16]{ravenelcob}. If we set 
$$C_*^{(k)} := K(n)_*[t_1, \ldots, t_k]/(v_nt_i^{p^n}-v_n^{p^i}t_i, 1
\leq i \leq k)$$ 
then $C_*^{(k)}$ is \'etale over $K(n)_*$ and $K(n)_*E(n)$ is the 
directed colimit of the $C_*^{(k)} $'s. 

The $K(n)_*$-B\"okstedt spectral sequence for $\thh(E(n))$ has as an 
$E^2$-term 
$$
\HH_*^{K(n)_*}(K(n)_*E(n)) \cong K(n)_*E(n)
$$ 
concentrated
in homological degree zero. Thus $K(n)_*\thh(E(n)) \cong K(n)_*E(n)$
and the isomorphism is induced by the map $E(n) \ra
\thh(E(n))$. Therefore, this map is a $K(n)$-equivalence and thus
$K(n)$-locally $\thh(E(n))$ is equivalent to $E(n)$. 
\end{proof}
We calculate $K(i)_*E(n)$ for $1 \leq i \leq n-1$ using the following
description of morphisms of graded commutative $BP_*$-algebras from $K(i)_*E(n)$
to some graded commutative ring $B_*$. For $n=2$ we had an argument
that was rather involved and Paul Goerss suggested the following
simpler proof. 

We consider the map $g\colon  BP_*BP \ra K(i)_*E(n)$ of graded commutative 
$\Z_{(p)}$-algebras given by 
$$ BP_*BP \ra K(i)_* \otimes_{BP_*} BP_*BP \otimes_{BP_*} E(n)_* \cong 
K(i)_*E(n)$$ 
which uses the canonical maps $BP_* \ra K(i)_*$ and $BP_* \ra E(n)_*$ and 
the isomorphism from \eqref{eq:kien}.  
By \cite[Theorem A2.1.27]{ravenelcob} this  map corresponds to a 
triple $((\eta_L)_*F_i, (\eta_R)_*G_n, f)$ where $\eta_L \colon  K(i)_* 
\ra K(i)_*E(n)$ is the left unit, $\eta_R \colon  E(n)_* \ra K(i)_*E(n)$ 
is the right unit and 
$(\eta_L)_*F_i$ and $(\eta_R)_*G_n$ are the $p$-typical formal group laws 
that are given by the corresponding change of coefficients. Here, $f$ is a 
strict isomorphism between the $p$-typical formal group laws 
$(\eta_L)_*F_i$ and $(\eta_R)_*G_n$ over $K(i)_*E(n)$. By 
\cite[Lemma A2.1.26]{ravenelcob} such a strict isomorphism is always of 
the form 
$$ f(x) = \sum_j{}^{(\eta_R)_*G_n} t_jx^{p^j}.$$

The $p$-series of the Honda formal group law $F_i$ is 
$$ [p]_{F_i}(x) = v_ix^{p^i}$$ 
and the same is true for $[p]_{(\eta_L)_*F_i}[x]$ because the left unit 
just embeds $K(i)_*$ into $K(i)_*E(n)$. The $p$-series of $(\eta_R)_*G_n$ is 
$$ [p]_{(\eta_R)_*G_n}(x) = w_1x^p +_{(\eta_R)_*G_n} 
\ldots +_{(\eta_R)_*G_n} w_nx^{p^n} $$
for $w_i = \eta_R(v_i)$.  
 
First, we state an elementary lemma about powers of $p$. 
\begin{lem} \label{lem:ppower}
Let $m\geq 2$, let $r, \ell_1, \ldots, \ell_m$ be natural numbers bigger or equal to
$1$, and assume that $\ell_j \neq \ell_k$ for $j \neq k$. Then 
$p^r$ cannot be written as a sum $p^{\ell_1} + \ldots + p^{\ell_m}$. 
\end{lem}
\begin{proof}
Assume 
$$ p^r = p^{\ell_1} + \ldots + p^{\ell_m}.$$
Without loss of generality let $\ell_1$ be minimal among the $\ell_j$'s. Then 
$$ p^r = p^{\ell_1}(1+ p^{\ell_2-\ell_1} + \ldots + p^{\ell_m-\ell_1}).$$
This is only possible if all the $\ell_j - \ell_1$ are equal to zero
and if $m = p^{r-\ell_1}$. But $\ell_j - \ell_1 = 0$ for all $2 \leq j
\leq m$ implies that all the $\ell_j$'s are equal to $\ell_1$ and this
contradicts our assumption. 
\end{proof}

\begin{prop} \label{prop:colimetale}
For all $1 \leq i \leq n$ $K(i)_*E(n)$ is a colimit of \'etale
$K(i)_*[w_{i+1},\ldots, w_n^{\pm 1}]$-algebras.   
\end{prop}
\begin{proof}
In the following we fix $i$ and $n$. 
We denote by $B(i,n)_*$ the graded commutative 
$K(i)_*$-algebra $K(i)_*[w_{i+1},\ldots, w_{n-1}, w_n^{\pm 1}]$. 
For a given $m \geq 1$ consider the graded commutative $BP_*$-subalgebra 
$BP_*[t_1, \ldots, t_m]$ of $BP_*BP$ and define 
$$
B_m =\mathrm{Image}\big(B(i,n)_*[t_1, \ldots, t_m] \ra K(i)_*E(n)\big)\,.
$$
Thus we can express $B_m$ as $B(i,n)_*[t_1,\ldots,t_m]/\sim$ where 
$\sim$ denotes the quotient that arises from the relations that the
$t_r$'s and $w_j$'s satisfy in $K(i)_*E(n)$. 
Note that $B_{m+1}$ is free as a $B_m$-module for all $m \geq 1$. Indeed, in each
step we adjoin a new polynomial generator $x$ to a graded commutative ring
$R_*$ that satisfies relations of the form $x^{p^r} - ux -y$ with a
unit $u \in R_*^\times$ and $y \in R_*$. 

The strict isomorphism $f(x) = \sum_j{}^{(\eta_R)_*G_n} t_jx^{p^j}$ satisfies 
$$ [p]_{(\eta_R)_*G_n}(f(x)) = f([p]_{(\eta_L)_*F_i}(x))$$
and this yields the equality 
\begin{equation} \label{eq:gnsums}
w_1(f(x))^p +_{(\eta_R)_*G_n} \ldots +_{(\eta_R)_*G_n}  
w_n(f(x))^{p^n} =
f(v_ix^{p^i}) = \sum_j{}^{(\eta_R)_*G_n} t_j(v_ix^{p^i})^{p^j}.
\end{equation}

On the right hand side in 
$\sum_j{}^{(\eta_R)_*G_n} t_j v_i^{p^j} x^{p^{i+j}}$  
the relations for the $t_r$ are detected by the powers 
$x^{p^{i+r}}$. Lemma \ref{lem:ppower} ensures
that for a given  $x^{p^{i+r}}$ we only have to consider the coefficient $t_j
v_i^{p^j}$ with $i+j=i+r$ coming from the linear term of the 
$(\eta_R)_*G_n$-sum 
$\sum_j{}^{(\eta_R)_*G_n} t_j v_i^{p^j} x^{p^{i+j}}$ and this is 
$t_{r}v_i^{p^{r}}$. 

As
the right hand side starts with $x^{p^i}$, it is a direct consequence
that $w_1, \ldots, w_{i-1}=0$ and from the coefficients of $x^{p^i}$
we obtain that $w_i = v_i$ in $K(i)_*E(n)$. 

We prove that $B_1$ is \'etale over $B(i,n)_*$ and that for every $m$, $B_{m}$ is 
\'etale over $B_{m-1}$. It follows that the algebras
$B_m$ are \'etale over $B(i,n)_*$.

Thus we have to show that the modules of relative K\"ahler differentials 
$\Omega^1_{B_{1}|B(i,n)_*}$ and $\Omega^1_{B_{m}|B_{m-1}}$ 
are trivial for all $m \geq 2$.

For $m=1$ we compare the coefficients of $x^{p^{i+1}}$ in
\eqref{eq:gnsums}. In this case only the linear terms of the
$(\eta_R)_*G_n$-sums contribute something and we obtain 
$$ v_it_1^{p^i} + w_{i+1}t_0 = t_1v_i^p$$
and therefore $t_1 = v_i^{-p}(v_it_1^{p^i} + w_{i+1})$. This gives a flat 
extension and the K\"ahler differential on $t_1$ is equal to 
$$ dt_1 = 0 + v_i^{-p}dw_{i+1}$$ 
and hence $B_1$ is \'etale over $B(i,n)_*$. 

Consider $B_m$. Then the first relation 
for $t_{m}$ is given by the relation of the coefficients for
$x^{p^{i+m}}$. 

We know that the formal group law $G_n(x,y)$ is of the form 
$$ G_n(x,y) = x + y + \sum_{i,j \geq 1} a_{i,j}x^iy^j$$
where the $a_{i,j} \in E(n)_* = \Z_{(p)}[v_1,\ldots,v_{n-1}, v_n^{\pm
  1}]$.  Equation~\eqref{eq:gnsums} relates power series with coefficients in
$K(i)_*E(n)$, hence the coefficients $\bar{a}_{i,j}$ of $(\eta_R)_*G_n$ are 
now considered in $K(i)_*E(n)$ and are elements of
$\F_p[w_i, \ldots, w_{n-1}, w_n^{\pm 1}]$. On the left hand side of
\eqref{eq:gnsums}  we get coefficients that involve some polynomials 
of $\bar{a}_{i,j}$'s, some $p$th powers of $t_j$'s and some expressions
in $w_k$'s. For $m+i \leq n$ we actually get a coefficient
$w_{m+i}t_0^{p^{m+i+0}} = w_{i+m}$.  

The $\bar{a}_{i,j}$'s are in $B(i,n)_*$, so they don't
contribute anything to the relative K\"ahler differentials. The
K\"ahler differentials on the $t_j^{p^k}$ are trivial because we are
over $\F_p$.  Hence we
can express the K\"ahler differential $dt_m$ up to a factor of $v_i^{p^m} =
w_i^{p^m}$ via K\"ahler differentials in the $w_k$'s. As $v_i^{p^m}$
is invertible in $B(i,n)_*$, the relative K\"ahler differentials  
$\Omega^1_{B_m|B_{m-1}}$ are  trivial for all $m \geq 1$. 
\end{proof}

\begin{thm}\label{thm:hh}
For all $1 \leq i \leq n$ we have an isomorphism of $K(i)_*E(n)$-algebras
$$ \HH_*^{K(i)_*}(K(i)_*E(n)) \cong K(i)_*E(n) \otimes_{\F_p}
\Lambda_{\F_p}(dw_{i+1},\ldots, dw_n).$$ 
\end{thm}
\begin{proof}
We have shown that $K(i)_*E(n)$ is the sequential colimit of the $B_m$'s. 
As the $K(i)_*$-algebras $B_m$ are \'etale
over $B(i,n)_*$ and as Hochschild homology commutes with
localization we can rewrite  $\HH_*(B_m)$ as 
\begin{align*}
\HH_*(B_m) \cong\ & B_m \otimes_{B(i,n)_*}
\HH_*^{K(i)_*}(B(i,n)_*)\\ 
\cong\ & B_m \otimes_{B(i,n)_*} (B(i,n)_* \otimes_{\F_p}
\Lambda_{\F_p}(dw_{i+1},\ldots, dw_n)) \\
\cong\ & B_m\otimes_{\F_p} \Lambda_{\F_p}(dw_{i+1},\ldots, dw_n))
\end{align*}
using \cite{wg} and the Hochschild-Kostant-Rosenberg theorem. 
Hochschild homology commutes with colimits, hence we obtain
$$ \HH_*^{K(i)_*}(K(i)_*E(n)) \cong \colim_m \HH^{K(i)_*}_*(B_m)\cong
K(i)_*E(n) \otimes_{\F_p} \Lambda_{\F_p}(dw_{i+1},\ldots, dw_n)\,.$$  
\end{proof}
\begin{thm} \label{thm:kithhen}
Assume that $p$ is an odd prime and that $E(n)$ is an $E_3$-ring
spectrum. Then, for all  $1 \leq i \leq n$, we have an isomorphism of
$K(i)_*E(n)$-algebras
$$ K(i)_*\thh(E(n)) \cong  K(i)_*E(n) \otimes_{\F_p}
\Lambda_{\F_p}(dw_{i+1},\ldots, dw_n).$$
\end{thm}
\begin{proof}
	We use the B\"okstedt spectral sequence \cite{Boe}, \cite[IX.2.9]{ekmm}, with $E^2$-term
$$ E^2_{r,s} = (\HH_r^{K(i)_*}(K(i)_*E(n)))_s\,,$$
where $r$ denotes the homological and $s$ the internal degree. 
By a result of Angeltveit and Rognes~\cite[Prop.~4.3]{ar05}, an $E_3$-structure
on $E(n)$ implies that this spectral is one of commutative
$K(i)_*E(n)$-algebras.
The
multiplicative generators $dw_j$ for $i \leq j \leq n$ sit in bidegree
$(1,2p^j-2)$ and hence they cannot carry any non-trivial
differentials. Therefore the spectral sequence collapses at the
$E^2$-term. As the abutment is a free graded commutative
$K(i)_*E(n)$-algebra, there cannot be any multiplicative
extensions. 
\end{proof}

\begin{rem}
	Note if $E(n)$ admits an $E_2$ structure, 
	the B{\"o}kstedt spectral sequence is one of
	$K(i)_*$-algebras by~\cite[Prop.~4.3]{ar05}. It therefore collapses
	since all $K(i)_*$-algebra generators lie in columns $0$ and $1$. This
	gives the same formula for
	$K(i)_*\thh(E(n))$ as a $K(i)_*$-module, but not as a
	$K(i)_*$-algebra, since there is now room for $K(i)_*$-algebra
	extensions.
\end{rem}

\section{Blue-shift for $\thh(E(n))$}
If we assume that $p$ is an odd prime and that $E(n)$ is an
$E_\infty$-ring spectrum, then $\thh(E(n))$ is a commutative
$E(n)$-algebra spectrum and the cofiber of the unit map 
$$ \bthh(E(n)) = \text{cofiber}\big(E(n) \ra \thh(E(n))\big)$$ 
is a non-unital commutative $E(n)$-algebra spectrum. 
If $E(n)$ carries an $E_3$-structure, then by \cite[\S 3.3]{bfv}, \cite{bm} the morphism
$E(n) \ra \thh(E(n))$ is an $E_2$-map. This implies the following
useful fact: 
\begin{lem} \label{lem:enlocal}
If $E(n)$ is an $E_3$-spectrum, then $\thh(E(n))$ is an $E(n)$-module
spectrum and in particular, $\thh(E(n))$ is $E(n)$-local. 
\end{lem}

Let $L_n$ denote the localization at $E(n)$, and in particular $L_0$ is the
rationalization. Recall that there is a
well-known chromatic fracture square  
$$ \xymatrix{
{L_nX} \ar[r] \ar[d] & {L_{K(n)}X} \ar[d]\\
{L_{n-1}X} \ar[r] & {L_{n-1}L_{K(n)}X.}
}$$
It is shown for instance in \cite[Example 3.3]{acb} and
\cite[Proposition 2.2]{bauer} that the homotopy 
pullback of 
$$ \xymatrix{
 & {L_{K(n)}X} \ar[d]\\
{L_{n-1}X} \ar[r] & {L_{n-1}L_{K(n)}X.}
}$$
is an $E(n)$-localization of $X$. The statement in \cite[Proposition
2.2]{bauer} is more general and \cite{acb} work out far more general
local-to-global statements.

We always know from Proposition \ref{prop:knlocal} that the unit
map is a $K(n)$-local equivalence. The  chromatic square for
$\bthh(E(n))$ is: 
$$ \xymatrix{
{\bthh(E(n)) =L_{K(n) \vee E(n-1)}\bthh(E(n))} \ar[r]
\ar[d]& {L_{K(n)}\bthh(E(n))} \ar[d]\\
{L_{E(n-1)}\bthh(E(n))} \ar[r] & {L_{E(n-1)}(L_{K(n)}\bthh(E(n)))\,.}
}$$
The $K(n)$-homology of $\bthh(E(n))$ is zero by Proposition~\ref{prop:knlocal}. It follows that
the localization $L_{K(n)}\bthh(E(n))$ is trivial, and hence
$L_{E(n-1)}(L_{K(n)}\bthh(E(n)))$ is also trivial. 
Therefore the vertical map on the left hand side is an
equivalence and we obtain a nice example of blue-shift:
\begin{lem} \label{lem:cofenminus1} If $E(n)$ is an $E_3$-spectrum,
  then the cofiber $\bthh(E(n))$ is $E(n-1)$-local. 
\end{lem}

\section{Topological Hochschild homology of $E(2)$} \label{sec:e2}
In this section, we discuss in more detail the topological Hochschild homology
of $E(2)$, which we will denote by $\etwo=E(2)$ to simplify the notation.
As explained in the proof of Lemma~\ref{lem:lambda}, the computations of
Theorem~\ref{thm:kithhen} for $E(2)$ can be expressed as follows:
\begin{align}\label{eq.5.12}
K(0)_*\thh(\etwo) & 
\cong K(0)_*\etwo \otimes \Lambda_\Q(dt_1, dt_2),\\
K(1)_*\thh(\etwo) &
\cong K(1)_*\etwo \otimes \Lambda_{\F_p}(dt_1),\\
K(2)_*\thh(\etwo) &
\cong K(2)_*\etwo.
\end{align}
Notice that these computations do not require the assumption that $\etwo$ is an
$E_3$-ring spectrum: for the rational case we have a commutative structure
anyhow, while in the $K(1)$ and $K(2)$ cases, the $E^2$ page of the B\"okstedt
spectral sequences is concentrated on columns 0 and 1 (respectively 0).
\begin{lem}\label{lem:lambda}
For $i=1,2$, there exist classes $\lambda_i\in \thh_{2p^i-1}(\etwo)$ with the
following properties. Under the Hurewicz homomorphism
\begin{itemize}
\itemr{a} the class $\lambda_i$ maps to $dt_i \in
K(0)_{2p^i-1}\thh(\etwo)$, for $i=1,2$; 
\itemr{b} the class $\lambda_1$ maps to $dt_1\in K(1)_{2p^2-1}\thh(\etwo)$.
\end{itemize}
\end{lem}

\begin{proof}
We use McClure-Staffeldt's computation of $\thh_*(BP)$
in~\cite[Remark~4.3]{mccs}, which has been validated by the proof~\cite{bm13} that $BP$
admits an $E_4$ structure. We briefly recall the computation. The integral,
rational and mod $p$ homology of $BP$ are given as
\begin{equation*}
H\Z_*BP\cong \Z_{(p)}[t_i\,|\,i\geq 1],
\ \ 
K(0)_*BP\cong \Q[t_i\,|\,i\geq 1]
\ \ \textup{ and }\ \ 
H{\F_p}_*BP\cong \Z[\bar\xi_i\,|\,i\geq 1],
\end{equation*}
where the class $t_i\in H\Z_{2p^i-1}BP$ maps to $\bar\xi_i$ under mod
$(p)$ reduction~\cite[Proof of Theorem 
5.2.8]{ravenelcob} and to the class with same name $t_i$ under rationalization.
The associated B\"okstedt spectral sequences collapse,
providing isomorphisms
\begin{align*}
H\Z_*\thh(BP) & 
\cong H\Z_*BP\otimes \Lambda_{\Z_{(p)}}(dt_i\,|\,i\geq1),\\
K(0)_*\thh(BP) & 
\cong K(0)_*BP\otimes \Lambda_{\Q}(dt_i\,|\,i\geq1)\ \ \textup{and}\\
H{\F_p}_*\thh(BP) & 
\cong H{\F_p}_*BP\otimes \Lambda_{\F_p}(d\bar\xi_i\,|\,i\geq1),
\end{align*}
with $dx=\sigma_*(x)$, where $\sigma\colon \Sigma BP\to\thh(BP)$ is the map
given in~\eqref{sigma}. 
There is an isomorphism
\begin{equation*}
\thh_*(BP)\cong BP_*\otimes \Lambda_{\Z_{(p)}}(\lambda_i\,|\,i\geq1),
\end{equation*}
and the Hurewicz homomorphism
\begin{equation*}
\thh_*(BP)\to H\Z_*\thh(BP)
\end{equation*}
is an inclusion mapping $\lambda_i$ to $dt_i$.
In particular, the classes $dt_i$ (integral and rational) and
$d\bar\xi_i$ are spherical: 
they are the image of $\lambda_i$ under the Hurewicz homomorphism mapping from
$\thh_*(BP)$.
For $i\geq 1$, let us define
\begin{equation*}
\lambda_i\in \thh_{2p^i-1}(\etwo)
\end{equation*}
as the image of the class with same name under the natural
map 
\begin{equation*}
\thh_*(BP)\to\thh_*(\etwo).
\end{equation*}
In the rational case, we have
\begin{equation*}
\eta_R(v_i)\equiv \alpha_it_i
\end{equation*}
modulo decomposables in $K(0)_*BP$, where $\alpha_i\in\Q$ is a unit.
We deduce that 
\begin{equation*}
K(0)_*\etwo\cong\Q[t_1,t_2][\eta_R(v_2)^{-1}]
\end{equation*}
and the B\"okstedt spectral sequence recovers
\begin{equation*}
K(0)_*\thh(\etwo)\cong K(0)_*\etwo\otimes \Lambda_\Q(dt_1, dt_2).
\end{equation*}
By naturality, comparing with the case of $BP$, we deduce that 
the Hurewicz homomorphism $\thh_*(\etwo)\to K(0)_*\thh(\etwo)$
maps $\lambda_i$ to $dt_i$.
\par
For $K(1)_*$-homology, we argue similarly, using the commutative square 
\begin{equation*}
\xymatrix{
	\thh_*(BP)\ar[r]\ar[d]&
	K(1)_*\thh(BP)\ar[d]\\
	\thh_*(\etwo)\ar[r]&
	K(1)_*\thh(\etwo).
}
\end{equation*}
We have $K(1)_*BP\cong K(1)_*[t_i\,|\,i\geq1]$, and
the B\"okstedt spectral sequence yields
\begin{equation*}
K(1)_*\thh(BP)\cong
K(1)_*BP\otimes \Lambda_{\F_p}(dt_i\,|\,i\geq1).
\end{equation*}
Comparing the B\"okstedt spectral sequences for $H{\Z}_*\thh(BP)$
and $K(1)_*\thh(BP)$, we deduce that the class $\lambda_1\in\thh_*(BP)$ maps to $dt_1\in
K(1)_*\thh(BP)$.
Recall that 
\begin{equation*}
K(1)_*\etwo=K(1)_*[t_i\,|\,i\geq 1][\eta_R(v_2)^{-1}]/(\eta_R(v_j)\,|\,{j\geq3})
\end{equation*} 
is a colimit of {\'e}tale algebras over $K(1)_*[w_2, w_2^{-1}]$, 
where
\begin{equation*}
w_2=\eta_R(v_2)=v_1^pt_1-v_1t_1^p.
\end{equation*}
In particular $dw_2=v_1^pdt_1$, and the B\"okstedt spectral sequence provides
the formula given above for $K(1)_*\thh(\etwo)$.
Now obviously $dt_1\in K(1)_*\thh(BP)$ maps to $dt_1\in K(1)_*\thh(\etwo)$. 
This implies assertion (b) of the lemma.
\end{proof}
\begin{rem}
  Note that the above proof does not require the map $BP \ra E(n)$ to be an
  $E_3$-map. 
\end{rem}
The class $\lambda_1\in\thh_{2p-1}(\etwo)$ of Lemma~\ref{lem:lambda}
corresponds to a map
$\lambda_1\colon  S^{2p-1}\to \thh(\etwo)$.
Smashing with $\etwo$, using the $\etwo$-module structure of $\thh(\etwo)$
(assuming an $E_3$ structure on $\etwo$), 
and composing with
the cofiber $\thh(\etwo)\to\bthh(\etwo)$ of the unit,
we obtain a map
\begin{equation*}
	j_1\colon \Sigma^{2p-1}\etwo\cong \etwo\wedge S^{2p-1}\to
\etwo\wedge\thh(\etwo)\to\thh(\etwo)\to\bthh(\etwo).
\end{equation*}
In the same fashion, we obtain a map $j_2\colon \Sigma^{2p^2-1}\etwo\to
\bthh(\etwo)$ corresponding to the class $\lambda_2$.
\begin{lem}
The map $j_1$ factors through a map
\begin{equation*}
	\bar j_1\colon \Sigma^{2p-1}L_1\etwo\to\bthh(\etwo)
\end{equation*}
that is a $K(1)_*$-isomorphism, and whose cofiber $C(\bar j_1)$ is 
a rational spectrum.
\end{lem}
\begin{proof}
Recall from Lemma~\ref{lem:cofenminus1} that the cofiber $\bthh(\etwo)$
of the unit map is $E(1)$-local. In particular, the map $j_1$ factors through a
map
\begin{equation*}
	\bar j_1\colon \Sigma^{2p-1}L_1\etwo\to\bthh(\etwo).
\end{equation*}
The localization map $\etwo\to L_1\etwo$
is a $K(1)_*$-isomorphism, and therefore so are the induced maps
$\ell\colon \thh(\etwo)\to\thh(L_1\etwo)$ and
$\bar\ell\colon \bthh(\etwo)\to\bthh(L_1\etwo)$, by convergence of the 
$K(1)$-based B\"okstedt spectral sequence. Hence, to prove the claim, it
suffices to show that the composition 
\begin{equation}\label{eq:comp}
\Sigma^{2p-1}L_1\etwo\xr{\bar j_1}\bthh(\etwo)\xr{\bar\ell}\bthh(L_1\etwo)
\end{equation}
is a $K(1)_*$-isomorphism. The $K(1)$-based B\"okstedt spectral sequence
for $L_1\etwo$ is identical to the one of $\etwo$, computed above as
\begin{equation*}
	E^2_{*,*}=K(1)_*\etwo\otimes\Lambda_{\F_p}(dt_1) \Rightarrow
        K(1)_*\thh(\etwo),	
\end{equation*}
where $K(1)_*\etwo$ is in filtration degree zero and $K(1)_*\etwo\{dt_1\}$ is in
filtration degree $1$, 
and where all differentials are zero.
By definition of the map $j_1$, if $1\in K(1)_0\etwo$ is the unit, then
${j_1}_*(\Sigma^{2p-1}1)$ is represented modulo lower filtration by the permanent
cycle $dt_1$ in $E^2_{1,*}$. Since this is a spectral sequence of
$K(1)_*\etwo$-modules, 
the composition $\eqref{eq:comp}$ induces a map in $K(1)$ homology that is
represented modulo lower filtration by the isomorphism 
$\Sigma^{2p-1}K(1)_*\etwo\to E^2_{1,*}=K(1)_*\etwo\{dt_1\}$
sending a class $\Sigma^{2p-1}w$ to $wdt_1$. 
It is therefore a
$K(1)_*$-isomorphism, proving the claim.

Now we consider the cofiber $C(\bar j_1)$ of $\bar j_1$, sitting in an
exact triangle 
\begin{equation}\label{eq:triangle}
	\Sigma^{2p-1}L_1\etwo\xr{\bar j_1}\bthh(\etwo)\xr{k} C(\bar
        j_1)\xr{\delta} \Sigma^{2p}L_1\etwo.
\end{equation}
Since $\bar j_1$ is a $K(1)_*$-isomorphism, we know that $K(1)_*C(\bar
j_1)=0$, and since 
$\bthh(\etwo)$ and thus $C(\bar j_1)$ are $E(1)$-local, we deduce (as in 
Lemma~\ref{lem:cofenminus1}) that $C(\bar j_1)$ is $E(0)$-local (\ie, rational).
\end{proof}
We now define a map $\lambda_{12}\colon L_0S^{2p^2-2p-2}\to C(\bar j_1)$ as a
composition over the cofibers
\begin{equation*}
L_0S^{2p^2-2p-2}\to L_0\thh(E) \to L_0\bthh(E)\to C(\bar j_1), 
\end{equation*}
where the first map above realizes the class $dt_1dt_2\in K(0)_*\thh(E)$.
Smashing $\lambda_{12}$ with $E$ and using the module structure we obtain a map
\begin{equation*}
	j_{12}\colon \Sigma^{2p^2-2p-2}L_0E\to C(\bar j_1). 
\end{equation*}
Similarly, $\lambda_2$ induces a map
\begin{equation*}
	j_{2}\colon \Sigma^{2p^2-1}L_0E\to C(\bar j_1). 
\end{equation*}
\begin{thm} \label{thm:thhofe2}
Let $p$ be an odd prime such that $\etwo=E(2)$, the second Johnson-Wilson
spectrum at $p$, is an $E_3$-ring spectrum.
Then the map $j_2\vee j_{12}$ lifts to a map 
\begin{equation*}
\bar j_2\vee\bar j_{12}\colon \Sigma^{2p^2-1}L_0E\vee\Sigma^{2p^2-2p-2}L_0E\to
\bthh(E)
\end{equation*}
and the sum $\beta$ of $\bar j_1$, $\bar j_2$ and $\bar j_{12}$
is a weak equivalence of $E$-modules
\begin{equation*}
\beta\colon 
\Sigma^{2p-1}L_1\etwo\vee
\Sigma^{2p^2-1}L_0\etwo\vee
\Sigma^{2p^2+2p-2}L_0\etwo\to
\bthh(\etwo).
\end{equation*}
\end{thm}
\begin{proof}
The composition $\delta\circ(j_2\vee j_{12})$ is trivial, so that
$ j_2\vee  j_{12}$ lifts to a map $\bar j_2\vee \bar j_{12}$:
\begin{equation*}
\xymatrix{ 
& 
\Sigma^{2p^2-1}L_0\etwo\vee \Sigma^{2p^2+2p-2}L_0\etwo
\ar@{.>}[dl]_{\bar j_2\vee \bar j_{12}}
\ar[d]^{j_2\vee j_{12}}
\ar[dr]^{\simeq *}
&
\\
\bthh(\etwo)
\ar[r]^-{k}
&
C(\bar j_1)
\ar[r]^{\delta}
&
\Sigma^{2p}L_1E.}
\end{equation*}
Indeed, $\Sigma^{2p} L_1\etwo$ fits in the chromatic fracture pullback diagram 
\begin{equation*}
\xymatrix{
	\Sigma^{2p} L_1\etwo
	\ar[r]
	\ar[d]
	& 
	\Sigma^{2p}L_{K(1)}\etwo
	\ar[d] 
	\\
	\Sigma^{2p}L_0\etwo
	\ar[r] 
	& 
	\Sigma^{2p}L_{0}( L_{K(1)} \etwo).
}
 \end{equation*}
The composition of $\delta\circ(j_2\vee j_{12})$ with the left
vertical map to $\Sigma^{2p}L_0\etwo$ is trivial, since it factors over 
the composition
\begin{equation*}
	L_0\bthh(\etwo) \xr{} L_0 C(\bar j_1) \xr{} \Sigma^{2p}L_{0} \etwo
\end{equation*}
of two consecutive maps in the ($E(0)$-localized) cofiber
sequence~\eqref{eq:triangle}. 
The composition of $\delta\circ(j_2\vee j_{12})$ with the top map to 
$\Sigma^{2p}L_{K(1)}\etwo$ is trivial as well; indeed, 
there is no non-trivial map from a $K(1)$-acyclic to a $K(1)$-local spectrum. 
This finishes the proof that $\delta\circ(j_2\vee j_{12})$ is trivial
and that the lift exists. We now define $\beta$ as the sum
\begin{equation*}
	\beta=\bar j_1\vee \bar j_2\vee \bar j_{12}\colon  \Sigma^{2p-1}L_1\etwo\vee\Sigma^{2p^2-1}
L_0\etwo\vee \Sigma^{2p^2+2p-2} 
L_0\etwo\to\bthh(\etwo).
\end{equation*}
Finally, we claim that $\beta$ is a $K(0)_*$-isomorphism:
this is analogous to the proof above that $\bar j_1$ is a $K(1)_*$-isomorphism,
working this time with the $K(0)$-based B\"okstedt spectral sequence.
Since $\beta$ is a $K(0)_*$- and a $K(1)_*$-isomorphism of $E(1)$-local
spectra, it is a weak equivalence.
\end{proof}
Assume now that in addition to $\etwo$ being an $E_3$-ring spectrum, 
the unit map $\etwo\to\thh(\etwo)$ splits in the homotopy category
(this holds for example if $\etwo$ is an $E_\infty$-ring spectrum). We
then have a weak 
equivalence of $E$-modules $\etwo \vee\bthh(\etwo)\to \thh(\etwo)$. 
On the other hand, summing $\beta$ with the identity of $\etwo$ gives a weak
equivalence
\begin{equation*}
\id\vee\beta\colon \etwo\vee
\Sigma^{2p-1}L_1\etwo\vee
\Sigma^{2p^2-1}L_0\etwo\vee
\Sigma^{2p^2+2p-2}L_0\etwo
\to
\etwo\vee
\bthh(\etwo).
\end{equation*}
This implies the following corollary of Theorem~\ref{thm:thhofe2}.
\begin{cor}\label{cor:thhe2}
Assume that $p$ is an odd prime, and that the second Johnson-Wilson spectrum
$\etwo=E(2)$ admits an $E_3$-structure.
If the unit map $\etwo\to\thh(\etwo)$ splits in the homotopy category, then 
the maps above provide a weak equivalence of $E$-modules
\begin{equation*}
\etwo\vee\Sigma^{2p-1}L_1\etwo\vee
\Sigma^{2p^2-1}L_0\etwo\vee
\Sigma^{2p^2+2p-2}L_0\etwo\to
\thh(\etwo).
\end{equation*}
\end{cor}
\begin{rem}
Corollary~\ref{cor:thhe2} implies that 
\begin{itemize}
	\item[-] the $2^0$ summand of $K(2)_*\etwo$ in $K(2)_*\thh(\etwo)$
		indexed by $1$,
	\item[-] the $2^1$ summands of $K(1)_*\etwo$ in $K(1)_*\thh(\etwo)$
		indexed by $1$ and $dt_1$, 
	\item[-] the $2^2$ summands of $K(0)_*\etwo$ in $K(0)_*\thh(\etwo)$
		indexed by $1$, $dt_1$, $dt_2$ and $dt_1dt_2$
\end{itemize}
assemble, in $\thh(\etwo)$,  into 
\begin{itemize}
	\item[-] the $2^0$ summand  $\etwo$ indexed by $1$ and detected
		by $K(0)_*$, $K(1)_*$ and $K(2)_*$, 
	\item[-] the $2^1-2^0$ summand $L_1\etwo$ indexed by $dt_1$ and
		detected by $K(0)_*$ and $K(1)_*$, and
	\item[-] the $2^2-2^1$ summands $L_0\etwo$ indexed by
		$dt_2$ and $dt_1dt_2$ and detected by $K(0)_*$.
\end{itemize}
Notice that Bruner and Rognes~\cite{BRunpub} obtain very
similar computations for $K(i)_*\thh(\mathrm{tmf})$ for $i=0,1,2$,
where $\mathrm{tmf}$ denotes the connective
spectrum of topological modular form.
\end{rem}
We can picture the summands of $\thh(\etwo)$ in a $2$-dimensional cube
of local pieces (up to suspensions, where $E=L_2E$): 
\begin{center}
	\renewcommand{\arraystretch}{2.4}
	\begin{tabular*}{13em}{ccc}
	&\multicolumn{1}{c}{$1$}&\multicolumn{1}{c}{$dt_1$}\\
	\cline{2-3}
	\multicolumn{1}{r}{$1$\kern 4pt} & \multicolumn{1}{|c|}{$\etwo$} &\multicolumn{1}{c|}{$L_1\etwo$}\\
	\cline{2-3}
	\multicolumn{1}{r}{$dt_2$\kern 4pt} &\multicolumn{1}{|c|}{$L_0\etwo$}&\multicolumn{1}{c|}{$L_0\etwo$}\\
	\cline{2-3}
\end{tabular*}
\renewcommand{\arraystretch}{1}
\end{center}
\par\medskip
\noindent
We conjecture that this picture extends to describe a decomposition of
$\thh(E(n))$ into $2^n$ summands,  with summands placed in an $n$-dimensional
cube, where the $i$th edge has two coordinates $1$ and $dt_i$.   
We formulate this as follows. 
\begin{conj}
If $p$ is an odd prime such that $E(n)$ is a sufficiently commutative $S$-algebra,
then $\thh(E(n))$ decomposes as a sum of $2^n$ factors, namely $2^{n-i-1}$
suspended copies of  $L_iE(n)$ for each $0\leq i\leq n-1$, plus one copy of
$E(n)$. More precisely, the $L_iE(n)$ summands are indexed by the
$2^{n-i-1}$ monomial generators
\begin{equation*}
\omega\in\Lambda_\Q(dt_1,\dots,dt_{n-i-1})\{dt_{n-i}\}\subset
K(0)_*\thh(E(n)),
\end{equation*}
and the summand corresponding to such a monomial $\omega$ is
$\Sigma^{|\omega|}L_iE(n)$.	
\end{conj}

\end{document}